\documentclass[12pt, reqno]{amsart}

\usepackage[margin=1.3 in]{geometry}
\usepackage[utf8]{inputenc}
\usepackage[T1]{fontenc}
\usepackage[english]{babel}

\usepackage{graphicx}
\usepackage{tikz-cd}
\usepackage{import}
\usepackage{wrapfig}
\usepackage{tabularx}

\usepackage{physics}

\usepackage{xcolor}
\definecolor{antiquefuchsia}{rgb}{0.57, 0.36, 0.51}
\definecolor{azure}{rgb}{0.0, 0.5, 1.0}

\usepackage[colorlinks = true, linkcolor = azure, citecolor = antiquefuchsia]{hyperref}

\usepackage{amsmath, amsthm, amssymb, amsfonts}
\usepackage{enumitem}
\usepackage{mathrsfs}

\usepackage{comment}

\theoremstyle{plain}
\newtheorem{theorem}{Theorem}
\newtheorem{proposition}[theorem]{Proposition}

\newtheorem{lemma}[theorem]{Lemma}

\theoremstyle{definition}
\newtheorem{definition}[theorem]{Definition}

\newtheorem{remark}[theorem]{Remark}

\numberwithin{theorem}{section}
\numberwithin{equation}{section}

\newcommand{\N}{\mathbb{N}} 
\newcommand{\R}{\mathbb{R}} 
\newcommand{\A}{\mathbb{A}} 
 
\newcommand{\B}{\mathcal{B}} 
 
\newcommand{\I}{\mathbb{I}} 
\newcommand{\X}{\mathcal{X}}
\newcommand{\J}{\mathbf{J}}
\newcommand{\loc}{\mathrm{loc}}
\renewcommand{\min}{\mathrm{min}}
\renewcommand{\max}{\mathrm{max}}
\renewcommand{\div}{\mathrm{div}}

\DeclareMathOperator{\supp}{supp}
\DeclareMathOperator{\diam}{diam}

\newcommand{\vertiii}[1]{{\left\vert\kern-0.25ex\left\vert\kern-0.25ex\left\vert #1 
    \right\vert\kern-0.25ex\right\vert\kern-0.25ex\right\vert}}

\title[BVM for the Calder{\'o}n problem with p.w.\,constant conductivities]{A Bernstein--von-Mises theorem for the Calder{\'on} problem with piecewise constant conductivities}
\author{Jan Bohr}
\address{University of Cambridge\footnote{{\it Department of Pure Mathematics and Mathematical Statistics}, Wilberforce Road,  Cambridge CB3 0WB, UK}}
\date{\today}
\email{bohr@maths.cam.ac.uk}

\begin{document}

\maketitle

\begin{abstract}
This note considers a finite dimensional statistical model for the Calder{\'o}n problem with piecewise constant conductivities. In this setting it is shown that injectivity of the forward map and its linearisation suffice to prove the invertibility of the information operator, resulting in a Bernstein--von-Mises theorem and optimality guarantees for estimation by Bayesian posterior means.
\end{abstract}

\section{Introduction}
This article is concerned with the statistical analysis of nonlinear inverse problems modelled on parameter spaces $\Theta\subset \R^D$ ($D\in \N$). For {\it mildly ill-posed} problems, recent years have seen great progress in establishing theoretical guarantees for estimation and uncertainty quantification in a high-dimensional setting, where $D\rightarrow \infty$, as the number of measurements increases; see e.g.~the articles \cite{KLS16,GK18,NVW18,MNP19a,MNP21,GN20,NW20} and \cite{
MNP20,
NiPa21,NW20,BoNi21}, which analyse inverse parameter identification problems for several representative partial differential equations (PDE). In the framework of Bayesian inverse problems it has been of particular interest to gain an understanding of the average `posterior landscape' and the latter group of references addresses this issue by approximating posteriors with simpler types of distributions. Key conditions here include invertibility results for the infinite-dimensional information operator to obtain Bernstein--von-Mises type approximations by Gaussians \cite{MNP21, NiPa21}, or the weaker requirement of {\it polynomial-in-$D$} eigenvalue lower bounds for the $D$-dimensional information operator to obtain relevant log-concave approximations \cite{NW20, BoNi21}. These conditions can be viewed as part of a spectrum with progressively weaker invertibility requirements on the information operator.

For {\it severely ill-posed} problems,  results comparable to the above do not exist in a high-dimensional setting ($D\rightarrow \infty$). This is exemplified by the notorious {\it Calder{\'o}n problem} (see the surveys in \cite{Bor03, U09, Uhl14}), which is well known to exhibit only logarithmic stability in an infinite-dimensional setting---as a consequence, non-parametric estimators converge no better than at logarithmic rates in the inverse noise level \cite{AbNi19} and high-dimension posterior approximations as above cannot be expected.

This brief note is concerned with the Calder{\'o}n problem for piecewise constant conductivities. This is a widely studied variant of the problem (see e.g.~\cite{KLMS08,DuSt16,Har19,AlSa21,AAS20}) that yields a model with parameter space of  {\it fixed dimension $D$}. Postponing precise definitions to \textsection \ref{EITsection}, our setting can be described as follows: We take the viewpoint of {\it Electrical Impedance Tomography (EIT)}, where one seeks to recover an unknown conductivity $\gamma$ inside a body $\Omega\subset \R^d$ from  voltage-to-current measurements $\Lambda_\gamma$ at the boundary $\partial \Omega$. 
We make the assumption that $\gamma$ is piecewise constant with respect to a given partition  $\Omega_1 \cup \dots \cup \Omega_D \subset \Omega $ and obeys bounds $0<\gamma_\min \le \gamma\le \gamma_\max$, such that it can be described by 
a parameter 
\begin{equation}
\theta \in \Theta = [\gamma_\min,\gamma_\max]^D\subset \R^D.
\end{equation} 
In applications this assumption can be  reasonable when part of the interior structure of $\Omega$ or the required resolution is known a priori. 
The operator $\Lambda_\gamma$ is then measured at a finite collection of electrodes $J_1,\dots,J_M\subset \partial \Omega$, assumed to be sufficently small and densely distributed (see also Figure \ref{figure}). This setting allows to consider the Calder{\'o}n problem in a classical i.i.d.\,noise model.

  \begin{figure}\label{figure}
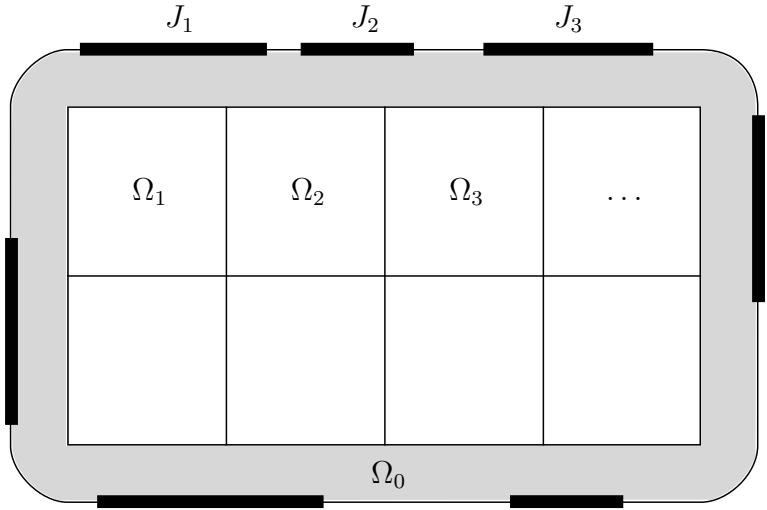   

\caption{The unknown conductivity $\gamma$ is assumed to be constant on $\Omega_1,\Omega_2,\dots$ and equal to $1$ in a region $\Omega_0$ near the boundary. Measurements are taken at electrodes $J_1,J_2,\dots$ on $\partial \Omega$.}
   \end{figure}

The contribution of this note lies in exhibiting, in the just described setting, estimators for $\theta$ with optimal convergence guarantees, as well as a Bernstein--von-Mises theorem,  that is, a Gaussian approximation result for Bayesian posterior distributions.
These results constitute final and definitive guarantees for the Bayesian method for EIT, advocated for in \cite{DuSt16}, in the setting of piecewise constant conductivities. In particular, as a consequence of the Bernstein-von--Mises theorem,  estimation of piecewise constant conductivities $\gamma$ by `posterior means' is {\it optimal} in an asymptotic minimax sense (see Remark \ref{mean})---at the same time, it is unclear whether classical reconstruction procedures (e.g.~based on Nachman's method \cite{KLMS07,KLMS08,KLMS09}) enjoy similar statistical optimality properties for piecewise constant conductivities.

 Our results also highlight the fact that high-dimensionality introduces significant challenges to the statistical understanding of inverse problems which are independent of the type of ill-posedness---indeed, as our results illustrate, many difficulties arising in the context of the first paragraph disappear for fixed dimension $D$, even for the severely ill-posed Calder{\'o}n problem. Further, for a given severely ill-posed problem, it is of significance to identify models in which the severe ill-posedness disappears and we do this here from a statistical point of view for the Calder{\'o}n problem by exhibiting a natural piecewise constant model.

 The statistical conclusions of this article heavily rely on an analytical observation that can be paraphrased as follows: if $\gamma$ is restricted to an appropriate  finite-dimensional space of conductivities, then  finitely many {\it noiseless} measurements of the operator $\Lambda_\gamma$ already suffice to stably determine $\gamma$. Statements of this form were proved e.g.~in \cite{Har19,AlSa21,AAS20}
  and in fact hold true for a larger class of inverse problems, essentially as long as the forward map (here $\gamma \mapsto \Lambda_\gamma$, mapping between appropriate spaces) and its linearisation are known to be injective. One aim of this article is to bridge the gap between the just mentioned advances in deterministic inverse problems and the statistical literature; this comes with the following message: if there are good reasons to model an inverse problem on a  parameter space of fixed finite dimension $D$, then under natural injectivity assumptions, the analysis of the noisy problem falls into a well-understood regime, where  excellent statistical guarantees -- including Le Cam's version of the Bernstein--von-Mises theorem -- become available.
At the same time we want to emphasise that outside of the severely ill-posed setting, where relevant stability constants typically grow exponentially in the model dimension $D$ (see also Remark \ref{exponentialgrowth}),
it is not necessary to fix the dimension $D$ to obtain relevant statistical guarantees---this is illustrated by the results in the first paragraph.

\subsection{Main results} \label{mainresult}

The statistical set-up of our main theorems fits into the general framework of \cite{MNP20,NiPa21,Nic22}; there the inverse problem is encoded in a forward map \begin{equation}\label{forwardmap}
G\colon \Theta\rightarrow L^2_\lambda(\X,V),\quad \theta \mapsto G_\theta
\end{equation}
and one seeks to make inference on $\theta$, given noisy measurements of $G_\theta$.
 Here  $L^2_\lambda(\X,V)$ is the space of $L^2$-integrable functions on a probability space $(\X,\lambda)$, taking values in a finite dimensional inner product space $(V,\langle \cdot ,\cdot\rangle)$. 
In the context of Calder{\'o}n's problem, we will have 
\begin{equation*}
\Theta \equiv[\gamma_\min,\gamma_\max]^D \subset \R^D,\quad \X = \{1,\dots,M\}, 
\quad \text{ and } \quad V=\R^M
\end{equation*}
for fixed $D\in\N$, $0<\gamma_\min <\gamma_\max$ and $M=M(D,\gamma_\min,\gamma_\max)\in \N$;  the precise definitions will be layed out in \textsection \ref{EITsection}, in particular, see Definition \ref{modeldefinition}. Note that while the Calder{\'o}n problem is usually posed in terms of an {\it operator valued} forward map $\gamma\mapsto \Lambda_\gamma$, the restriction to a $D$-dimensional parameter space allows for a  vector valued formulation as in \eqref{forwardmap}. In the context of electrical impedance tomography, one can think of $\X$ as a index set for a collection of electrodes, and of $G_\theta(x)$ as list of current measurements, when a standard voltage is applied at the electrode with index $x\in \mathcal{X}$. 

We consider a random design regression model with Gaussian noise: For $\theta \in \Theta$ define $P_\theta$ as the law of $(Y,X)\in \R^M\times \mathcal{X}$, where $Y$ satisfies the regression equation
\begin{equation*}
Y=G_\theta(X) +\epsilon,
\end{equation*}
with $X \sim \lambda$\,(\,$ = $\,uniform measure) and $\epsilon \sim \mathcal{N}_M(0,I)$ drawn independently. For EIT, this describes a statistical experiment, where an electrode is chosen at random (draw of $X$) and currents are measured at all $M$ electrode locations (recorded in $Y$); again we refer to \textsection \ref{EITsection} for details. Independent repetition of this procedure yields a sequence of statistical experiments
\begin{equation*}
\mathcal{P}_N=\{P^N_\theta: \theta \in \Theta\},\quad N=1,2,\dots,
\end{equation*}
where $P_\theta^N=P_\theta \otimes \dots \otimes P_\theta$ is the $N$-fold product measure on $(\R^M\times \mathcal{X})^N$. The corresponding expectation operator is denoted $E_\theta^N$ and we use the notation $Z_N=\left((Y_i,X_i):i=1,\dots, N\right)$ for a random vector of law $P^N_\theta$ for some $\theta\in \Theta$. Finally, for $N=1$ we write $\mathcal{P}=\mathcal{P}_1$ and $Z\equiv (Y,X) = Z_1$.

The experiment $\mathcal{P}$ (arising from any sufficiently smooth $G$, in particular in the case of the Calder{\'o}n problem) satisfies a natural regularity condition, called {\it differentiability in quadratic mean}, which allows the definition of {\it score function} and {\it information operator}, denoted by
\begin{equation*}
\A_\theta(Y,X)\in \left(\R^D\right)^*\quad \text{ and } \quad \N_\theta\in \R^{D\times D},\quad \theta\in \Theta\backslash \partial \Theta,
\end{equation*}
respectively. For precise definitions in the present context we refer to Definition \ref{definfo} and for a general overview of the importance of $\N_\theta$ in estimation problems we refer to the monograph \cite{vdV98}.

\begin{theorem}\label{informationinvertible} Assume that $G$ (and hence $\mathcal{P}$) arise from the Calder{\'o}n problem for piecewise constant conductivities as described in \textsection \ref{EITsection}. Then the information matrix $\N_\theta$ is invertible for all $\theta\in \Theta\backslash \partial \Theta$.
\end{theorem}


The proof of Theorem \ref{informationinvertible} is a consequence of the general theory in \cite{Har19,AlSa21,AAS20} and can be adapted to parametric models of other inverse problems, as long as injectivity and compactness properties similar to \ref{e1},\ref{e2},\ref{e3} from \textsection \ref{EITsection} are available.

Amongst the consequences of Theorem \ref{informationinvertible} are optimal consistency and uncertainty quantification results in a Bayesian setting. The Bayesian approach to inverse problems, see  \cite{S10,DS16,DuSt16} as well as the recent lecture notes \cite{Nic22}, uses the data $Z_N$ to update a chosen prior measure $\Pi$ on $\Theta$ to a posterior measure $\Pi(\cdot\vert Z_N)$. The posterior can in principle be computed from evaluating the forward map $G$ alone and is then used to make inference on $\theta$ by deriving point estimators and credible regions.

As prior we may choose here any Borel probability measure $\Pi$ on $\Theta$ with continuous and positive Lebesgue density $\pi$, that is,
\begin{equation}\label{prior}
d\Pi(\theta)=\pi(\theta) d\theta,\quad \text{ for some } \pi\in C(\Theta,(0,\infty)).
\end{equation}
Given data $Z_N$ in $(\R^M\times \mathcal{X})^N $, define the log-likelihood function
\begin{equation*}
\ell_N(\theta\vert Z_N) = - \frac 12 \sum_{i=1}^N \vert G_\theta(X_i) - Y_i\vert_{\R^M}^2,\quad \theta\in \Theta,
\end{equation*}
and the resulting {\it posterior measure} by
\begin{equation*}
\Pi(A\vert Z_N) = \frac{\int_A e^{\ell_N(\theta\vert Z_N)}d\Pi(\theta)}{\int_\Theta e^{\ell_N(\theta\vert Z_N)}d\Pi(\theta)},\quad A\subset \Theta \text{ Borel measurable}.
\end{equation*}

Our next result is a Bernstein--von-Mises theorem for the Calder{\'o}n problem:

\begin{theorem}\label{bvmthm} Assume that $G$ (and hence $\mathcal{P}_N$) arise from the Calder{\'o}n problem for piecewise constant conductivities as described in \textsection \ref{EITsection}. Let $\Pi$ be a prior as in \eqref{prior} and let $\theta_0 \in \Theta\backslash \partial \Theta$. Then, 
 if $\vartheta\sim \Pi(\cdot \vert Z_N)$
is a posterior draw, as $N\rightarrow \infty$,
\begin{equation*}
\Big \Vert \mathcal{L}\left(\sqrt N(\vartheta - \Psi_{\theta_0,N})\Big\vert Z_N\right) - \mathcal{N}_D(0,\N^{-1}_{\theta_0}) \Big \Vert_{\mathrm{TV}} \rightarrow 0 \text{ in } P_{\theta_0}^N.
\end{equation*}
Here $\mathcal{L}(\cdot)$ denotes the law of a random variable in $\R^D$, $\Vert \cdot \Vert_{\mathrm{TV}}$ is the total variation norm of Borel measures on $\R^D$, and the re-centring $\Psi_{\theta_0,N}$ is given by
\begin{equation}
 \Psi_{\theta_0,N}=\theta_0 + \frac 1N \sum_{i=1}^N \N_{\theta_0}^{-1} \A_{\theta_0}^T(Y_i,X_i).
\end{equation}
\end{theorem}

\begin{remark}[Posterior means]\label{mean} If the prior $\Pi$ is assumed to have finite moments to all orders (e.g.~if $\Pi$ is a uniform prior), then -- arguing as in Step VII of the proof of Theorem 8 in \cite{Nic20} -- the theorem also holds true when the  re-centring $\Psi_{\theta_0,N}$ is replaced by the more tractable posterior mean
\begin{equation*}
\bar \theta_N \equiv E^\Pi[\vartheta\vert Z_N].
\end{equation*}
By the same proof,  for $Z_N\sim P_{\theta_0}^N$ ($\theta_0\in \Theta\backslash \partial \Theta$) and as $N\rightarrow \infty$, 
 \begin{equation}\label{nonuniform}
	\sqrt N(\bar \theta_N - \theta_0) \rightarrow \mathcal{N}_D(0,\mathbb N_{\theta_0}^{-1})\quad \text{ in distribution},
 \end{equation}
 see also Theorem 10.8 in \cite{vdV98}.
 In particular, posterior means achieve the asymptotic minimax optimal variance for estimation of $\theta$ under square loss. That is, the posterior means $T_N=\bar \theta_N$ obtain the lower bound in the asymptotic risk estimate 
 \begin{equation} \label{minimax}
\lim_{\delta\rightarrow 0}\, \liminf_{N\rightarrow \infty}\, \sup_{\Vert \theta - \theta_0\Vert<\delta} \mathrm{Cov}^N_{\theta} \big [\sqrt N(T_N - \theta) \big] \ge \mathbb N_{\theta_0}^{-1},
 \end{equation}
 understood in the sense of the Loewner order, which is valid for all estimator sequences $(T_N\colon (\R^M\times \mathcal X)^N \rightarrow \Theta \colon N=1,2,\dots)$ (cf.~Theorem 8.11 in \cite{vdV98}). This follows from a version of \eqref{nonuniform} that is locally uniform in $\theta_0\in \Theta\backslash \partial \Theta$---we omit a detailed proof.
\end{remark}

Let us discuss some further consequence of the preceding theorem. First, for any $\theta_0\in \Theta\backslash \partial \Theta$ there exists a constant $C>0$ such that, as $N\rightarrow \infty$,
\begin{equation*}
\Pi\left(\theta: \Vert \theta - \theta_0 \Vert \le C N^{-1/2} \vert Z_N \right) \rightarrow 1 \text{ in } P_{\theta_0}^N.
\end{equation*}
This means that the posterior is {\it consistent} at $\theta_0$, in the sense that it is likely to concentrate its mass in small balls around the truth, at the {\it optimal} rate of $N^{-1/2}$. Recall from \cite{AbNi19} that in infinite dimensions (say, for $C^\beta$-regular conductivities for some $\beta>0$), the posterior only concentrates about $\theta_0$ with rate $(\log N)^{-\delta}$ for some $\delta>0$.

Moreover, using the strengthened version of the theorem from Remark \ref{mean}, one obtains an uncertainty quantification  as follows: For a confidence level $0<\alpha<1$, define {\it credible regions} $C_N\subset \Theta$  and quantiles $R_N>0$ by
\begin{equation*}
C_N=\{\theta\in \Theta: \vert \theta - \bar \theta_N\vert<R_N \},\quad \Pi(C_N\vert Z_N)=1-\alpha.
\end{equation*}
Note that finding $C_N$ and $R_N$ does {\it not} require knowledge of the information matrix $\N_{\theta_0}$, but can be computed approximately via Markov chain Monte Carlo methods.
Arguing as in \textsection 2.5 of \cite{MNP21} (see also \textsection 4.1.3 in \cite{Nic22}), the Bernstein-von-Mises theorem implies that, as $N\rightarrow \infty$,
\begin{equation*}
P_{\theta_0}^N(\theta_0\in C_N) \rightarrow 1-\alpha,
\end{equation*}
showing that $C_N$ has valid frequentist coverage of the true parameter $\theta_0$ at level $\alpha$.

\subsection{A statistical Calder{\'o}n problem} \label{EITsection} We now give a brief account of the Calder{\'o}n problem and explain how it can be recast as an inverse problem with a forward map $G$ as in \eqref{forwardmap}. Throughout, we employ the following notational conventions: 

\begin{itemize} 
\item If not stated otherwise, function spaces contain $\R$-valued functions and all Banach spaces are real;
\item for two Banach spaces $E$ and $F$, the space of bounded linear operators is denoted $\B(E,F)$ and the operator norm is written $\Vert \cdot \Vert_{\B(E,F)}$ or $\Vert \cdot \Vert_{E\rightarrow F}$. If $E=F$, then we write $\B(E)$ instead of $\B(E,E)$;
\item the indicator function of a set $S$ is denoted with $1_S$.
\end{itemize}

\medskip

 Let $\Omega \subset \R^d$ ($d\ge 2$) be a bounded domain with smooth boundary $\partial \Omega$. Let $L_+^\infty(\Omega)$ be the space of $\gamma\in L^\infty(\Omega)$ which satisfy a lower bound $\gamma\ge c$ for some $c\in (0,\infty)$, almost everywhere in $\Omega$. Given $\gamma\in L^\infty_+(\Omega)$, the boundary value problem
\begin{equation*}
\div (\gamma \nabla u)=0 \text{ in }\Omega,\quad u = g \text{ on } \partial \Omega
\end{equation*}
has a unique solution $u=u^g$ for any given boundary datum $g\colon \partial \Omega\rightarrow \R$ of sufficient regularity. The {\it Dirichlet-to-Neumann operator}  of $\gamma$ is then formally defined by \begin{equation*}
\Lambda_{\gamma}g = \gamma \partial_\nu u^g\vert_{\partial \Omega},
\end{equation*} where $\partial_\nu$ is the (outward pointing) normal derivative.
 Denoting with $H^s=H^s(\partial \Omega)$ the Sobolev space of order $s\in \R$, the operator $\Lambda_\gamma$ extends to a bounded linear map $H^{1/2}\rightarrow H^{-1/2}$.  This gives rise to a {\it nonlinear} map 
\begin{equation*}
L_+^\infty(\Omega)\rightarrow \B(H^{1/2},H^{-1/2}),\quad \gamma \mapsto \Lambda_\gamma,
\end{equation*}
and Calder{\'o}n's  problem asks to establish injectivity of this map on large classes of conductivities, or indeed to reconstruct $\gamma$ from $\Lambda_\gamma$. We refer to \cite{Cal80,KoVo85,SyUh87,Nac88,Nac96,AsPa06} for some of the most influential articles on this topic.

In electrical impedance tomography, the function $\gamma$ models an unknown electrical conductivity inside $\Omega$ and $\Lambda_\gamma g$ is the electrical current caused by applying voltage $g$ on $\partial \Omega$; the conductivity $\gamma$ is then to be determined from voltage-to-current measurements. We want to pose a statistical inverse problem of the following form: After fixing electrode positions $J_1,\dots, J_M\subset \partial \Omega$, one electrode $J_{X}$ is chosen at random and a constant voltage $g=1_{J_X}$ 
is applied. The resulting current is then measured at all electrodes, corrupted by statistical noise, yielding a data vector  $Y$ in $\R^M$. The random vector $(Y,X)\in \R^M\times \{1,\dots, M\}$ can then be interpreted as data vector in a statistical regression model, as described in \textsection \ref{mainresult}.

To pass to a setting as in \textsection \ref{mainresult}, we make several {\it a priori} assumptions, which will inform the precise definition of the forward map $G$ in \eqref{forwardmap}. In essence, we require that $\gamma$ is piecewise constant and  satisfies a bound $\gamma_\min\le \gamma \le \gamma_\max$. In this case, the electrodes $J_1,\dots, J_M$ can be chosen such that $\gamma$ is determined from finitely many {\it noiseless} measurements at these positions (cf.\,Remark \ref{remdelta}); this is the setting in which our statistical analysis takes place (see also Figure \ref{figure} above).

We now describe the assumptions on $\gamma$ and the forward map $G$ in detail.
First, we assume that $\gamma$ belongs to a known, finite dimensional class $E_D$ of conductivities.   Let $\Omega_0,\dots,\Omega_D\subset \Omega$ be a collection of pairwise disjoint domains, satisfying
\begin{equation*}
\bar\Omega = \bar \Omega_0\cup\dots\cup \bar \Omega_D\quad \text{ and } \quad \bar \Omega_i \subset \Omega\quad \text{ for } i=1,\dots, D,
\end{equation*}
and having piecewise smooth boundaries.
Denoting the associated characteristic functions by $1_{\Omega_0},1_{\Omega_1},\dots$, we define
\begin{equation}\label{defe}
E_D :=\left \{\gamma\in L^\infty(\Omega): \gamma=\gamma_\theta \equiv 1_{\Omega_0} + \sum_{i=1}^D \theta_i 1_{\Omega_i}\text{ for } (\theta_1,\dots,\theta_D)\in \R^D\right \},
\end{equation}
with {\it tangent space} $E'_D=\{\kappa\in L^\infty(\Omega): \kappa + 1_{\Omega_0} \in E_D\}$.
The space $E_D$ satisfies the following properties (see Theorem \ref{checkproperties}): 
\begin{enumerate}[label=(P\arabic*)]
	\item \label{e1} $E_D\cap L_+^\infty(\Omega)\ni\gamma \rightarrow \Lambda_\gamma$ is injective;
	\smallskip
	\item \label{e2} For all $\gamma \in E_D\cap L_+^\infty(\Omega)$, the differential $E_D'\ni \kappa\mapsto d\Lambda_\gamma(\kappa)$ is injective.
\end{enumerate}	
 To formulate a third property, recall that an operator $T$ on $\partial \Omega$ is called {\it smoothing} if it extends to a bounded linear operator $T\colon H^s\rightarrow H^t$ for all choices of $s,t\in \R$. As functions in $E_D$ are constant near $\partial \Omega$, also the following holds true:

	
\begin{enumerate}[label=(P\arabic*)]
\setcounter{enumi}{2}
	\item \label{e3} For all $\gamma \in E_D\cap L_+^\infty(\Omega)$ and $\kappa \in E_D'$, the operator $d\Lambda_\gamma(\kappa)$ is smoothing and the map $(\gamma,\kappa)\mapsto \Vert d\Lambda_\gamma(\kappa) \Vert_{H^s\rightarrow H^t}$ is locally bounded for any $s,t\in \R$.
\end{enumerate}

\begin{remark}
The preceding three properties 
also hold for other finite dimensional classes $E_D$ (e.g.~containing piecewise analytic functions \cite{Har19}) and we could have chosen any such class; for the sake of concreteness we will only work with piecewise constant conductivities as in \eqref{defe}.
\end{remark}

\begin{remark} Property \ref{e3} guarantees certain compactness properties that allow for the passage to `finite measurements' in the sense of \cite{AlSa21}. One can dispense with \ref{e3}, if one passes to a different measurement model and uses e.g.~the {\it Neumann-to-Dirichlet} operators -- this route is chosen in the just cited article.
\end{remark}

Now suppose that $\J=\{J_1,\dots,J_M\}$ is a collection of pairwise disjoint, measurable subsets  $J_1,\dots, J_M\subset \partial \Omega$ of positive measure. We define the quantity 
\begin{equation}\label{Delta}
\Delta(\J):=\Big\vert \partial \Omega \backslash \bigcup_{k=1}^M J_k \Big\vert^{1/2} + \sup_{k=1,\dots, M} \diam (J_k)>0, 
\end{equation}
where $\vert \cdot \vert$ denotes the surface measure on $\partial \Omega$ and $\diam J = \sup_{x,y\in J}\vert x - y \vert$ is the diameter. Given a priori bounds $\gamma_\min \le \gamma \le \gamma_\max$, it will suffice to test $\Lambda_\gamma$ at the indicator functions $1_{J_1},\dots,1_{J_M}$, provided $\Delta(\J)$ is sufficiently small. Precisely, we make the assumption that
\begin{equation}\label{Deltacondition}
\Delta(\J) \le \delta = \delta(E_D,\gamma_\min,\gamma_\max),
\end{equation}
where $\delta>0$ is specified in Theorem \ref{smalldelta} below. For technical reasons we now pass to the {\it normalised} Dirichlet-to-Neumann operator \begin{equation}\label{normaliseddtn}
\tilde \Lambda_\gamma = \Lambda_\gamma - \Lambda_1,
\end{equation} which extends to an $L^2$-bounded  (in fact,  smoothing) operator and contains the same information as $\Lambda_\gamma$; this ensures that the $L^2$-pairing $\langle \tilde \Lambda_{\gamma} 1_{J_i},1_{J_j}\rangle$ is well defined.

\begin{definition}\label{modeldefinition} Given {\it a priori} choices of $D\in \N$ and $0<\gamma_\min\le \gamma_\max$, as well as a choice of $\J=\{J_1,\dots, J_M\}$ satisfying \eqref{Deltacondition}, we make the following definitions: 

\begin{equation*}
\Theta :=\{\theta \in \R^D\colon   \gamma_\min \le \theta_i \le \gamma_\max \text{ for all } i=1,\dots,D\},
\end{equation*}
\begin{equation}\label{defG}
G^{ij}_\theta := \frac{1}{(\vert J_i\vert \vert J_j\vert)^{1/2}}\langle \tilde \Lambda_{\gamma_\theta} 1_{J_i},1_{J_j}\rangle_{L^2(\partial \Omega)},\quad 1\le i,j,\le M,\theta \in \Theta,
\end{equation}
where $\gamma_\theta$ is the piecewise constant conductivity from \eqref{defe}. Finally, with $(\X,\lambda)=\left(\{1,\dots, M\},\text{Uniform}\right)$, a forward map $G$ as in \textsection\ref{mainresult} is defined by
\begin{equation}\label{defG2}
G\colon \Theta\rightarrow L^2_\lambda(\X,\R^M),\quad  G_\theta(x)=(G^{x j}_\theta:j=1,\dots, M)\in \R^M \quad (x\in \X).
\end{equation}
\end{definition}

\begin{remark}\label{remdelta}
In order for the results of  \textsection \ref{mainresult} to apply,  the electrodes have to be chosen in accordance with requirement \eqref{Deltacondition}, which is of course difficult to verify in practice. The threshold $\delta$ will in particular depend on the precise Lipschitz stability constants for the forward map and its linearisation on $E_D$, as well as a domain dependent constant---see also Remark 4 and Example 4 in \cite{AlSa21}.
\end{remark}

\section{Analytical aspects}\label{analytical aspects}

In this section we collect some results, mostly well known, on the forward map $\gamma\mapsto \Lambda_\gamma$ and its linearisation $d\Lambda_\gamma$. In particular, we show that properties \ref{e1},\ref{e2} and \ref{e3} from \textsection \ref{EITsection} are satisfied on spaces of piecewise constant conductivities. 

Throughout, $\Omega$ is a bounded, smooth domain $\Omega\subset \R^d$. We will write $\int_{\Omega}$ and $\int_{\partial \Omega}$ for integration on $\Omega$ and $\partial \Omega$, understood with respect to the Lebesgue measures of dimension $d$ and $d-1$, respectively; the corresponding $L^2$-pairing (as usual, extended as duality pairing between distributions) is denoted with $\langle\cdot,\cdot\rangle$ in both cases.

\subsection{Linearisation}\label{linearisation} The linearisation of $\gamma\mapsto \Lambda_\gamma$ was already computed in  Calder{\'o}n's original article \cite{Cal80}, we summarise his result (supplemented by a statement on Fr{\'e}chet differentiablity, which is easy to check) in the following proposition.
\begin{proposition}
Given $g\in H^{1/2}(\partial \Omega)$, the map $\gamma\mapsto \langle \Lambda_\gamma g ,g\rangle$ is real analytic on $L^\infty_+(\Omega)$. Its differential at $\gamma\in L^\infty_+(\Omega)$ in direction $\kappa\in L^\infty(\Omega)$ is given by
\begin{equation}\label{defdlambda}
\langle d\Lambda_\gamma(\kappa)g,g\rangle := \frac{d}{dt}\Big \vert_{t=0} \langle \Lambda_{\gamma+t\kappa} g,g\rangle =  \int_\Omega \kappa \vert \nabla u^g \vert^2,
\end{equation}
where $u=u^g$ is the unique solution in $H^1(\Omega)$ to \begin{equation*}
\div (\gamma \nabla u ) = 0  \text{ on }  \Omega \quad \text{ and }\quad u=g \text{ on } \partial \Omega.
\end{equation*}
Equation \eqref{defdlambda} defines a bounded linear operator $d\Lambda_\gamma$ from $L^\infty(\Omega)$ to $\B(H^{1/2},H^{-1/2})$, depending continuously on $\gamma\in L_+^\infty(\Omega)$. Moreover, the map $\gamma\mapsto \Lambda_\gamma$ is Fr{\'e}chet differentiable as map from $L_+^\infty(\Omega)$ into $\B(H^{1/2},H^{-1/2})$, with Fr{\'e}chet derivative $d\Lambda_\gamma$.\qed
\end{proposition}

The next proposition is an instance of a well-known phenomenon: If $\gamma$ agrees with a smooth conductivity $\gamma'$ in a neighbourhood of $\partial \Omega$, then $\Lambda_
\gamma - \Lambda_{\gamma'}$ is a smoothing operator. For the linearisation $d\Lambda_\gamma$ this has the following consequence:

\begin{proposition}\label{smoothing}
Let $K\subset \Omega$ be compact and $s,t\in \R$. Then for all $\gamma\in L_+^\infty(\Omega)$ with $\gamma = 1$ on $\Omega\backslash K$ and all $\kappa\in L^\infty(\Omega)$ with $\kappa=0$ on $\Omega\backslash K$ the differential $d\Lambda_\gamma(\kappa)$ extends to a bounded linear operator from $H^s$ to $H^t$, with operator norm
\begin{equation}
\Vert d\Lambda_{\gamma}(\kappa)\Vert_{H^s\rightarrow H^t} \le C
\end{equation}
for a constant $C$ depending only on $s,t,K,\Vert \kappa\Vert_{L^\infty(\Omega)}$ and $\inf_\Omega \gamma$.
\end{proposition}

For this, we require two auxiliary lemmas.

\begin{lemma}\label{poissondiff} For all $s\in (-\infty,1/2]$ and all $\gamma\in L_+^\infty(\Omega)$ with $\gamma=1$ on $\Omega\backslash K$ there exists a continuous linear Poisson operator 
\begin{equation*}
P_\gamma\colon H^s(\partial\Omega)\rightarrow H^{s+1/2}(\Omega)\cap H^1_{\loc}(\Omega),
\end{equation*} 
as follows: If $g\in H^s(\partial \Omega)$, then $u=P_\gamma g$ solves the boundary value problem  
\begin{equation}\label{bvp}
\div(\gamma \nabla u) =0 \text{ in } \Omega,\quad  u = g\text{ on } \partial \Omega
\end{equation} in a sense made precise in the proof.
Moreover, the difference operator $P_\gamma-P_1$ maps $H^s(\partial \Omega)$ into $H^1_0(\Omega)$, with norm
\begin{equation*}
\Vert P_\gamma - P_1 \Vert_{H^s(\partial \Omega)\rightarrow H^1_0(\Omega)} \le C(s,K,\inf_\Omega \gamma).
\end{equation*}
\end{lemma}

This is essentially proved in \cite[Lemma 19]{AbNi19}; for the convenience of the reader, we give a brief recap of the proof, referring there for further details.

\begin{proof}[Proof of Lemma \ref{poissondiff}]
For $\gamma=1$, the existence and mapping properties of $P_1$ follow from the standard solution theory of the Laplace equation, as summarised e.g.~in Remark 7.2 of \cite{LiMa72}. For general $\gamma\in L_+^\infty(\Omega)$ with $\gamma=1$ in $\Omega\backslash K$, we say that $u\in H^{s+1/2}(\Omega)\cap H^1_\loc(\Omega)$ is a solution of \eqref{bvp}, if $w=u-P_1 g$ satisfies
\begin{equation*}
w\in H_0^1(\Omega),\quad \langle \gamma \nabla w,\nabla \varphi\rangle = F(\varphi)\text{ for all } \varphi\in H_0^1(\Omega),
\end{equation*}
where $F(\varphi) = \langle -\gamma \nabla P_1g,\nabla \varphi\rangle $ for $\varphi \in C_c^\infty(\Omega)$. One shows, using the mapping properties of $P_1$ and that  $\supp(1-\gamma)\subset K$, that $F$ extends to a continuous linear functional on $H_0^1(\Omega)$ with norm $\Vert F \Vert_{H^{-1}(\Omega)} \le C \Vert g \Vert_{H^s(\partial \Omega)}$, where $C=C(s,K,\inf_\Omega \gamma)>0$. Existence and a priori bounds on $w$ then follow from  Lax-Milgram theory and the lemma is proved upon setting $P_\gamma g = w + P_1g$.
\end{proof}

The next lemma is an application of elliptic regularity theory for the Laplace equation, see Remark 7.2 in \cite{LiMa72} and Lemma 22 in \cite{AbNi19}.

\begin{lemma}\label{harmonic} For  $U\subset \Omega$ open and $s\in \R$, consider the space of harmonic functions $\mathcal{H}^s(U)=\{u\in H^s(U): \Delta u = 0 \}$, equipped with the $H^s(U)$ norm. Then:
\begin{enumerate}[label=(\roman*)]
\item\label{poissoniso} for all $s\in \R$, the operator $P_1\colon H^s(\partial \Omega)\rightarrow \mathcal{H}^{s+1/2}(\Omega)$ is an isomorphism;
\item \label{restrictionharmonic} for all $s,t\in \R$ and any open set $U\subset \bar U \subset \Omega$, restriction $u\mapsto u\vert_{U}$ defines a bounded linear map 
$
\mathcal{H}^s(\Omega)\rightarrow  \mathcal{H}^t(U).
$\qed
\end{enumerate}
\end{lemma}

\begin{proof}[Proof of Proposition \ref{smoothing}]
Let $g,h\in C^\infty(\partial \Omega)$ and suppose that $U$ is an open set with $K\subset U \subset \bar U \subset \Omega$. Consider $u=P_\gamma g$ and  $v= P_\gamma h$, both elements in $H^1(\Omega)$. Then by the H{\"o}lder inequality,
\begin{equation*}\label{smoothing1}
 \vert \langle d\Lambda_\gamma(\kappa)g,h\rangle \vert =\vert  \langle \kappa \nabla u ,\nabla v\rangle\vert \le \Vert \kappa \Vert_{L^\infty(\Omega)} \Vert u \Vert_{H^1(U)}\Vert v \Vert_{H^1(U)},
\end{equation*}
because $\supp \kappa \subset U$. Then, for any $s\in (\infty,1/2]$ we have
\begin{eqnarray*}
\Vert u \Vert_{H^1(U)} &\le & \Vert (P_\gamma - P_1)g\Vert_{H^1(\Omega)} + \Vert P_1g \Vert_{H^1(U)}\\
&\le & C \left\{\Vert g \Vert_{H^s(\partial \Omega)} + \Vert P_1g \Vert_{H^{s+1/2}(\Omega)} \right\}\\
&\le & C' \Vert g \Vert_{H^s(\partial \Omega)} 
\end{eqnarray*}
with constants $C,C'>0$, depending at most on $s,K,U,$ and $\inf_\Omega \gamma$. Here we used Lemma \ref{poissondiff} and Lemma \ref{harmonic}\ref{restrictionharmonic} for the second estimate and part \ref{poissoniso} of that lemma for the third one. Bounding $v$ in the same way, we see that for all $s,r\in (-\infty,1/2]$,
\begin{equation}
\vert \langle d\Lambda_\gamma(\kappa)g,h\rangle \vert \le C'' \Vert g \Vert_{H^s(\partial\Omega)}\Vert h \Vert_{H^r(\partial\Omega)},
\end{equation}
with $C''$ depending at most on $s,r,K,U$ and $\inf_\Omega \gamma$. This shows that $d\Lambda_\gamma(\kappa)$ extends to a bounded operator from $H^s(\partial\Omega)$ to $H^{-r}(\partial \Omega)$, with norm $\le C''$. This encompasses the desired mapping property for all $s,t\in \R$.
\end{proof}

\subsection{Piecewise constant conductivites} 
\begin{theorem}\label{checkproperties} The space $E_D$ of piecewise constant conductivites from \eqref{defe} satisfies properties \ref{e1},\ref{e2} and \ref{e3}.
\end{theorem}

\begin{proof}[Proof of \ref{e1}] \renewcommand{\qedsymbol}{}
Injectivity results on piecewise constant conductivities go back at least to \cite{KoVo85}; for a setting that encompasses piecewise smooth boundaries $\partial \Omega_0,\dots,$ $ \partial \Omega_D$ we refer to \cite[Theorem 2.7]{AlVe05}. (In $d=2$, article \cite{AsPa06} settles the injectivity question on {\it all of} $L_+^\infty(\Omega)$ and regularity assumptions on $\partial \Omega_0,\dots, \partial \Omega_D$ become obsolete.)
\end{proof}

\begin{proof}[Proof of \ref{e2}]  \renewcommand{\qedsymbol}{}
We use the technique of \cite{LeRi08}, where linearised injectivity was proved in a slightly different context.  Assume that $d\Lambda_\gamma(\kappa)=0$ for some $\gamma \in E_D\cap L^\infty_+(\partial \Omega)$ and some $\kappa \in E'_D$ with $\kappa\neq 0$. Then by a simple geometric argument, and after replacing $\kappa$ with $-\kappa$ if necessary, there exists a subdomain $U\subset \Omega$
such that
\begin{equation}
0\neq \kappa \ge 0 \text{ on } U,\quad \bar U \cap \partial \Omega \neq \emptyset \quad \text{ and } \quad  U \text{ is connected}.
\end{equation}
Let $V_1\subset U$ be a small ball on which $\kappa>0$ and put $V_2=\Omega\backslash U$. Then $V_1$ and $V_2$ satisfy the properties of \cite[Theorem 2.7]{Geb08} and as any $\gamma\in E_D$ has the unique continuation property, the just cited theorem yields a sequence $(u_n)\subset H^1(\Omega)$ of solutions to $\div(\gamma \nabla u)=0$ on $\Omega$, with energies satisfying $\int_{V_1} \vert \nabla u_n \vert^2 \rightarrow \infty$ and $\int_{V_2} \vert \nabla u_n\vert^2 \rightarrow 0$ as $n\rightarrow \infty$.  This leads to a contradiction, as $g_n=u_n\vert_{\partial \Omega} \in H^{1/2}(\partial \Omega)$ satisfies
\begin{equation*}
0 = \langle d\Lambda_\gamma(\kappa)g_n,g_n\rangle = \int_\Omega \kappa \vert \nabla u_n \vert^2 \ge \min_{V_1}\kappa \int_{V_1} \vert \nabla u_n\vert^2 - \Vert \kappa \Vert_\infty \int_{V_2} \vert \nabla u_n\vert^2 \rightarrow \infty.
\end{equation*}
Here we used the representation of $d\Lambda_\gamma(\kappa)$ from Proposition \ref{linearisation}.
\end{proof}

\begin{proof}[Proof of \ref{e3}] This is a direct consequence Proposition \ref{smoothing} above.
\end{proof}

\section{Finite number of electrodes}

This section is based on the articles \cite{AlSa21,AAS20}, which develop a general theory of infinite-dimensional inverse problems with finite (noiseless) measurements. In our context, this will allow to formulate sufficient {\it a priori} conditions on the set of electrodes $\mathbf{J}$ in terms of the quantity $\Delta(\mathbf{J})$.

\subsection{Projection operators}\label{projsec} Let $\mathbf{J}=\{J_1,\dots, J_M\}$ be a collection of pairwise disjoint, measurable subsets of $\partial \Omega$ of positive measure. Writing $L^2$ for $L^2(\partial \Omega)$, we define the following projection operators:
\begin{eqnarray}
&P_\mathbf{J}\colon L^2(\partial \Omega)\rightarrow L^2(\partial \Omega),\quad 
&P_{\bf J}g=\sum_{k=1}^M \frac{1_{J_k}}{\vert J_k\vert} \int_{J_k} g,\\
&Q_{\bf J}\colon \B(L^2)\rightarrow \B(L^2),\quad &Q_{\bf J}T =  P_{\bf J}T P_{\bf J} \label{defqj}
\end{eqnarray}
We seek out norm bounds for $Q_\J$ and $I-Q_\J$ in terms of $\Delta(\J)$ from \eqref{Delta}. To simplify the statements below, we assume that the norms $\Vert \cdot \Vert_{H^s}$ on $H^s(\partial \Omega)$ are chosen to be increasing in $s$, and such that
\begin{equation}\label{sobemb}
\Vert g \Vert_{\infty} + \sup_{x,y\in \partial \Omega, x\neq y}\frac{g(x)-g(y)}{\vert x - y \vert} \le  \Vert g \Vert_{H^s},\quad \text{ for }g\in H^s(\partial \Omega),  s\ge d/2+1,
\end{equation}
which is possible by the Sobolev embedding theorem.

\begin{lemma}\label{qnorm} For all $T\in \B(L^2)$ we have
\begin{equation*}
\Vert Q_\J T \Vert_{L^2\rightarrow L^2}^2 \le \sum_{i,j=1}^M T_{ij}^2,\quad \text{ where } T_{ij}=\frac{1}{(\vert J_1 \vert \vert J_j\vert)^{1/2}}\langle T 1_{J_i},1_{J_j}\rangle_{L^2}.
\end{equation*}
\end{lemma}

\begin{proof}
Note that the functions $\psi_i=1_{J_i}/\vert J_i\vert^{1/2}$ ($i=1,\dots, M$) are orthonormal and that $P_\J g = \sum_{i=1}^M \psi_i \langle g,\psi_i\rangle_{L^2}$ for all $g\in L^2$. Hence, by Cauchy-Schwarz,
\begin{eqnarray*}
\Vert Q_\J T g\Vert_{L^2}^2 &=&\sum_{i=1}^M \langle \psi_i,TP_\J g\rangle_{L^2}^2 = \sum_{i=1}^M  \left( \sum_{j=1}^M \langle \psi_i,T\psi_j\rangle_{L^2} \langle g,\psi_j\rangle_{L^2} \right)^2\\
&\le & \sum_{i=1}^M \left(\sum_{j=1}^M T_{ij}^2\right)\left(\sum_{j=1}^M \langle g,\psi_j\rangle^2\right) \le \Vert g \Vert_{L^2}^2 \sum_{i,j=1}^M T_{ij}^2. 
\end{eqnarray*}
\end{proof}

\begin{lemma}\label{qerror} Let $s_*=d/2+1$.

\begin{enumerate}[label=(\roman*)]
\item \label{proj1} Let $s\ge s_*$ and suppose that $g\in H^s(\partial \Omega)$. Then
\begin{equation*}
\Vert (I-P_{\bf J}) g\Vert_{L^2(\partial \Omega)} \le  (1+\vert \partial \Omega\vert )^{1/2}    \Vert g \Vert_{H^s} \cdot \Delta({\bf J})
\end{equation*}
\item \label{proj2} Suppose that $T$ is a smoothing operator on $\partial \Omega$, such that
\begin{equation*}
\vertiii{T} =  (1+\vert \partial \Omega\vert )^{1/2}  \left(\Vert T \Vert_{L^2\rightarrow H^{s*}} +  \Vert T \Vert_{H^{-s_*}\rightarrow L^2}\right)  <\infty.
\end{equation*}
Then
\begin{equation*}
\Vert( I - Q_\J )T \Vert_{\B(L^2)} \le \Delta(\J)\cdot \vertiii{T}.
\end{equation*}
\end{enumerate}

\end{lemma}

\begin{proof}
Let $J_0=\partial \Omega\backslash \bigcup_{k=1}^M J_k $ and choose points $x_k\in J_k$ such that $\int_{J_k} g  = g(x_k) \vert J_k \vert$ for all $k=1,\dots,M$. The indicator functions of $J_0,\dots,J_M$ are orthogonal and hence
\begin{equation*}
\Vert (I-P_{\bf J}) g\Vert_{L^2(\partial \Omega)}^2 = \Vert  g\Vert_{L^2(J_0)}^2 + \sum_{k=1}^M \Vert g- g(x_k) \Vert_{L^2(J_k)}^2.
\end{equation*}
For $x\in J_k$ we have $\vert g(x) - g(x_k)\vert \le \diam(J_k)\cdot \Vert g \Vert_{H^s}$ by \eqref{sobemb} and thus we can bound the previous display by
\begin{equation*}
\le \vert J_0 \vert \cdot \Vert g \Vert_\infty^2 + \Vert g \Vert_{H^s}^2 \cdot \sup_{k=1,\dots, M} \diam(J_k)^2 \cdot \sum_{k=1}^M \vert J_k\vert \le  (1+\vert \partial \Omega\vert ) \cdot \Delta(\J)^2 \cdot \Vert g \Vert_{H^s}^2,
\end{equation*}
which proves \ref{proj1}. Next, note that $\Vert P_J \Vert_{L^2\rightarrow L^2} \le 1$ and write $c_\J = (1+\vert \partial \Omega \vert)^{1/2} \Delta(\J)$. Then for $g\in L^2(\partial \Omega)$ it holds that
\begin{eqnarray*}
\Vert (I-Q_\J)T g \Vert_{L^2} &\le &\Vert (I-P_\J) T g \Vert_{L^2} + \Vert P_JT(I-P_\J)g\Vert_{L^2}\\
& \le & c_\J \Vert T g \Vert_{H^{s*}} + \Vert T \Vert_{H^{-s^*}\rightarrow L^2} \Vert (I-P_\J)g\Vert_{H^{-s^*}}.
\end{eqnarray*}
Now $\Vert (I-P_\J)g\Vert_{H^{-s^*}} = \sup \langle (I-P_\J)g,h\rangle_{L^2}$, where the supremum is taken over $h\in H^{s_*}(\partial \Omega)$ with $\Vert h \Vert_{H^{s_*}}\le 1$. Using self-adjointness of $(I-P_\J)$, the Cauchy-Schwarz inequality and part \ref{proj1}, this can be bounded by $c_\J \Vert g \Vert_{L^2}$. Hence the previous display continues with
\begin{equation*}
\le c_\J \left(\Vert T \Vert_{L^2\rightarrow H^{s_*}} + \Vert T \Vert_{H^{-s_*}\rightarrow L^2}\right) \cdot \Vert g \Vert_{L^2},
\end{equation*}
as desired.
\end{proof}

\subsection{Stability estimates and finite measurements} In this section $E_D\subset L^\infty(\Omega)$ is an arbitrary $D$-dimensional linear subspace.

On $E_D$, the stability estimates required in the setting of \cite{AlSa21,AAS20} are immediate corollaries of the corresponding injectivity results. This is a consequence of the compactness argument from \cite{Bou13} (extended to non-convex compacts in \cite{AAS20}); precisely:

\begin{proposition} \label{stabilityforfree} Suppose properties \ref{e1} and \ref{e2} are satisfied and $K\subset E_D\cap L^\infty_+(\Omega)$ is compact.
Then for all $\gamma,\gamma' \in K$ and $\kappa\in E'_D$ it holds that
\begin{eqnarray}
\Vert \gamma - \gamma'\Vert_{\infty} &\le & S_1  \Vert \Lambda_{\gamma} - \Lambda_{\gamma'} \Vert_{H^{1/2}\rightarrow H^{-1/2}}\\
\Vert \kappa \Vert_{\infty} & \le & S_2 \Vert d\Lambda_\gamma(\kappa) \Vert_{H^{1/2}\rightarrow H^{-1/2}}
\end{eqnarray}
for constants $S_1,S_2>0$ only depending on $K$ and $E_D$.
\end{proposition}

\begin{proof}
The nonlinear stability estimate is provided by Theorem 1 in \cite{AAS20}, noting that $\gamma \mapsto \Lambda_\gamma$ is of class $C^1$ in view of Proposition \ref{linearisation}. The linear stability estimate follows from \ref{e2}, by taking an infimum over the compact set $K\times \{\kappa\in E_D: \Vert \kappa\Vert\le 1\}$.
\end{proof}

\begin{remark}\label{exponentialgrowth}
	Note that the stability constants $S_1$ and $S_2$ are obtained by compactness arguments and are thus not quantitative---however, they should be expected to grow exponentially in $D$ as $D\rightarrow \infty$, as verified in some special cases in \cite{Ron05}.
\end{remark}

The following contains the main theorem from \cite{AlSa21}, applied to our setting:

\begin{theorem}\label{smalldelta} In the setting of the previous proposition, assume in addition that $K$ is convex and that property \ref{e3} is satisfied. Then:
\begin{enumerate}[label=(\roman*)]

\item\label{smalldelta1} there exists constant $A=A(K)>0$ such that  $Q_\J$ from \eqref{defqj} satisfies
\begin{equation*}
\sup_{\gamma\in K} \Vert (1-Q_\J)d\Lambda_\gamma \Vert_{E'_D\rightarrow \B(L^2)} \le  A(K) \cdot \Delta(\J);
\end{equation*}
\item \label{smalldelta2} with $S_1,S_2$ and $A$ as in Proposition \ref{stabilityforfree} and part \ref{smalldelta1}, define
\begin{equation*}
\delta(E_D,K):= \frac{1}{2A\,\max(S_1,S_2)}
\end{equation*}
and suppose that $\Delta(\mathbf{J})\le \delta (E_D,K)$. Then  for all $\theta, \theta'
\in \Theta$ with $\gamma_\theta,\gamma_{\theta'}\in K$ (defined as in \eqref{defe}) the matrices $G_\theta,G_{\theta'}$ from \eqref{defG} satisfy
\begin{equation}\label{finitemeasurements1}
\Vert \theta - \theta' \Vert_{\infty} \le 2S_1 \Vert G_\theta - G_{\theta'} \Vert_{\R^{M\times M}}.
\end{equation}
Moreover, if $\theta \in \Theta$ with $\gamma_\theta \in K$ and $h\in \R^D$, then
\begin{equation}\label{finitemeasurements2}
\Vert h \Vert_\infty \le 2 S_2 \Vert d G_\theta(h) \Vert_{\R^{M\times M}},
\end{equation}
where $dG_\theta(h) = Q_\J d\Lambda_{\gamma_\theta}(h_11_{\Omega_1}+\dots, h_D1_{\Omega_D})$, viewed as element in $\mathrm{range}(Q_N)\equiv \R^{M\times M}$.
\end{enumerate}
\end{theorem}

\begin{proof}
By Lemma \ref{qerror}, for all  $\gamma \in E_D$ and $\kappa \in E_D'$ it holds that
\begin{equation*}
\Vert (I- Q_\J)d \Lambda_\gamma(\kappa)\Vert_{\B(L^2)} \le   \Delta(\J) \cdot \vertiii{d\Lambda_\gamma(\kappa)}.
\end{equation*}
Set $A(K)= \sup \vertiii{d\Lambda_{\gamma}(\kappa)}$, where the supremum is taken over all $(\gamma,\kappa)$ with $\gamma \in K$ and $\Vert \kappa \Vert \le 1$; this is finite by property \ref{e3} and \ref{smalldelta1} is proved.

Next, by Theorem 2(ii) in \cite{AlSa21}, if $\Delta(\J) \le 1/(2AS_1)$ and $\gamma,\gamma' \in K$, then 
\begin{equation*}
\Vert \gamma- \gamma' \Vert_{L^\infty} \le 2 S_1 \Vert Q_\J \tilde \Lambda_\gamma - Q_\J \tilde  \Lambda_{\gamma'} \Vert_{L^2\rightarrow L^2}
\end{equation*}
If $\gamma = \gamma_\theta$ and $\gamma'=\gamma_{\theta'}$, then the norm on the right hand side equals $\Vert G_\theta - G_{\theta'} \Vert_{\R^{M\times M}}$ by Lemma \ref{qnorm} and \eqref{finitemeasurements1}
follows.
Let $\kappa = h_1 1_{\Omega_1} +
\dots + h_D1_{\Omega_D} $. Then by Lemma \ref{qnorm} we have
\begin{equation*}
\Vert dG_\theta(h) \Vert_{\R^{M\times M}} \ge \Vert Q_\J d\Lambda_\gamma(\kappa) \Vert_{\B(L^2)} \ge \Vert d \Lambda_\gamma(\kappa)\Vert_{\B(L^2)} - \Vert (I-Q_\J)d\Lambda_\gamma (\kappa)\Vert_{\B(L^2)}.
\end{equation*}
By Proposition \ref{stabilityforfree} and part \ref{smalldelta1}, this can be bounded from below by
\begin{equation*}
\ge \left( S_2^{-1} - A \Delta(\J)\right) \Vert h \Vert_\infty
\end{equation*}
and if $\Delta(\J) \le 1/(2AS_2)$, we get the desired estimate \eqref{finitemeasurements2}.
\end{proof}

\begin{remark} In fact, the preceding proof requires a slight variant of \cite[Theorem 2]{AlSa21}, differing as follows: First, it is {\it not} necessary for the forward map to be globally Lipschitz if one is only interested in stability for $\gamma,\gamma'\in K$ -- in this case the additional approximation step in the proof becomes obsolete. Second, Hypothesis 1 of that theorem is only needed to obtain part (i) of the theorem and we achieve the corresponding bound by means of Lemma \ref{qerror}. Part (ii) of the theorem, which is needed above, is proved independently and (in their notation) for fixed $N$ such that $s_N\le 1/2C$. We apply the result for fixed $\mathbf{J}$ such that $A\Delta(\mathbf{J})\le 1/2S_1$, which is nothing but a change in notation.

\end{remark}

\section{Proof of the main theorems} 
In this section, we connect to the statistical setting and prove the main theorems.

\subsection{DQM property \& information operator}Suppose a statistical experiment $\mathcal{P}=\left(P_\theta:\theta\in \Theta\right )$ with parameter space $\Theta\subset \R^D$ describes a random design regression model with Gaussian noise, arising from a forward map
\begin{equation*}
G\colon \Theta\rightarrow L^2_\lambda(\X,V),\quad \theta\mapsto G_\theta
\end{equation*}
as in \textsection\ref{mainresult} -- at this point there is no need to specify to the Calder{\'o}n problem. Then differentiability of $G$ can be leveraged to show that $\mathcal{P}$ is {\it differentiable in quadratic mean (DQM)} (that is, \eqref{dqm3} below). For $h\in \R^D$ we write $\Vert h \Vert$ for any of the equivalent norms of $\R^D$.

\begin{proposition}\label{dqm} Assume that $U=\sup_{\theta \in \Theta} \Vert G_\theta \Vert_\infty<\infty$ and let $\theta$ be an interior point of $\Theta$. Suppose there exists a bounded linear operator $\I_\theta\colon \R^D\rightarrow L^2_\lambda(\X,V)$ such that
for all $x\in \mathcal{X}$ we have
\begin{equation}\label{dqm1}
\vert G_{\theta+h}(x) - G_\theta(x) -\I_\theta[h]\vert_V = o(\Vert h \Vert)\quad \text{ as } \Vert h \Vert \rightarrow 0.
\end{equation}
Suppose further that there exists constants $\epsilon, B>0$ such that 
\begin{equation}\label{dqm2}
\Vert G_{\theta + h} - G_\theta \Vert_\infty \le B \Vert h \Vert \quad \text{ for } \Vert h \Vert < \epsilon.
\end{equation}
Then $\mathcal{P}$ satisfies the DQM property at $\theta$ in the following sense: Define
\begin{equation}\label{score}
\mathbb{A}_\theta\colon \R^D\rightarrow L^2(V\times \X,P_\theta),\quad \mathbb{A}_\theta[h](y,x)=\langle y - G_\theta(x) ,\I_\theta[h](x)\rangle_V,
\end{equation}
then
\begin{equation}\label{dqm3}
\int_{V\times \X}\left[dP_{\theta+h}^{1/2} - dP_\theta^{1/2} - \frac12 \mathbb{A}_\theta[h] dP^{1/2}_\theta \right]^2 =o(\Vert h \Vert^2),\quad \text{ as } \Vert h \Vert \rightarrow 0.
\end{equation}
\end{proposition}

\begin{definition}\label{definfo} The map $\mathbb{A}_\theta\colon \R^D\rightarrow L^2(V\times \X, P_\theta)$ from \eqref{score} is called {\it score operator} of $\mathcal{P}$ at $\theta$. Further, the {\it information operator/matrix} of $\mathcal{P}$ at $\theta$ is defined as
\begin{equation*}
\N_\theta = \A_\theta^*\A_\theta \in \B(\R^D)\equiv \R^{D\times D},
\end{equation*}
where the adjoint of $\A_\theta$ is formed with respect to the natural inner products on $\R^D$ and $L^2(V\times \X,P_\theta)$.
\end{definition}

\begin{remark}\label{vdv}
For $(Y,X)\sim P_\theta$ we may consider $\A_\theta(Y,X)$ as random element in $(\R^D)^*$ (the dual space of $\R^D$, say, viewed as `row vectors'). The transpose $\A_\theta^T(Y,X)\in \R^D$ equals the score $\dot \ell_\theta\equiv \dot \ell_\theta(Y,X)$ from \cite{vdV98}; another description of the information matrix is thus
\begin{equation*}
\N_\theta = E_\theta[\dot \ell_\theta(Y,X)\dot\ell_\theta^T(Y,X)],
\end{equation*}
which coincides with the definition on p.39 of \cite{vdV98}.
\end{remark}

\begin{remark}\label{nfromi} Proposition \ref{dqm} extends to the infinite-dimensional setting from \cite{NiPa21} and yields a result similar to Theorem 1 in that article -- the difference is that here we consider total instead of partial derivatives.
 Moreover, Proposition 1 of \cite{NiPa21} also applies here, yielding yet another representation of the information matrix: With $\I_\theta \colon \R^D\rightarrow L^2_\lambda(\X,V)$ as in \eqref{dqm1}, it holds that
\begin{equation*}
\N_\theta = \I_\theta^*\I_\theta,
\end{equation*}
again taking adjoints with respect to the natural inner products.
\end{remark}

\begin{proof}[Proof of Proposition \ref{dqm}]
Let
\begin{eqnarray*}
f_h(y,x) &:=& \log\left( dP_{\theta+ h }^{1/2}/dP_\theta^{1/2}(y,x)\right)\\
&=& \frac{1}{2} \langle y, G_{\theta + h}(x) - G_\theta(x)\rangle_V - \frac 14 \left(\vert G_{\theta+h}(x) \vert_V^2 - \vert G_\theta(x)\vert_V^2\right).
\end{eqnarray*}
The proposition is proved, once we show that the integral
\begin{equation*}
\int_{V\times \X}\left[\frac{e^{f_h} -1 - \frac 12 \mathbb{A}_\theta[h]}{\Vert h \Vert} \right]^2 dP_\theta
\end{equation*}
vanishes in the limit $\Vert h \Vert \rightarrow 0$. For fixed $(y,x)\in V\times \mathcal{X}$, the integrant vanishes in the limit, as $h\mapsto \exp(f_h(y,x))$ is Fr{\'e}chet differentiable at $0$ with derivative $\frac 12 \A_\theta(y,x)$; this follows from \eqref{dqm1} and the chain rule. Further, by the convexity of $\exp$, the Cauchy-Schwarz inequality and property \eqref{dqm2}, it holds for all $0<\Vert h \Vert <\epsilon$ that
\begin{equation*}
\frac{\vert e^{f_h(y,x)} - 1 \vert}{\Vert h \Vert} \le \exp\left(\frac {\vert f_h(y,x) \vert}{\Vert h \Vert}  \right)  - 1 \le \exp\left(\frac B 2 \vert y \vert_V+ \frac 12 U^2  \right), 
\end{equation*}
which is $L^2$-integrable with respect to $P_\theta$.
Also $\A_\theta[h]$ is in $L^2(P_\theta)$, with norm $\lesssim \Vert h \Vert$. The result  then follows by the dominated convergence theorem.
\end{proof}

\subsection{Proof of the main theorems} We make the identifications
\begin{equation*}
\mathrm{range}(P_\mathbf{J})\equiv \R^M,\quad \mathrm{range}(Q_\mathbf{J})\equiv \R^{M\times M} \equiv L_\lambda^2(\mathcal{X},\R^M),
\end{equation*}
where the projections $P_\J$ and $Q_\J$ are defined as in \textsection \ref{projsec} and the $L^2$-space is as in Definition \ref{modeldefinition}. With $\gamma_\theta$ as in \eqref{defe}, the forward map $G$ can then be written as
\begin{equation*}
G\colon \Theta\rightarrow L^2_\lambda(\mathcal{X},\R^M),\quad G_\theta = Q_\J \tilde \Lambda_{\gamma{\theta}},
\end{equation*}
where $\tilde \Lambda_\gamma$ is the normalised Dirichlet-to-Neumann operator from \eqref{normaliseddtn}. By Propositon \ref{linearisation} it follows that $\theta\mapsto G_\theta$ is differentiable in the sense of Propostion \ref{dqm}, with
\begin{equation*}
\I_\theta(h)\equiv dG_\theta(h)= Q_\J d\Lambda_{\gamma_\theta}(\kappa),\quad \kappa=h_11_{\Omega_1}+\dots + h_D1_{\Omega_D}.
\end{equation*}
In particular, the model $\mathcal{P}$ defined in \textsection \ref{mainresult} is differentiable in quadratic mean and score and information operators are well defined.

\begin{proof}[Proof of Theorem \ref{informationinvertible}]
By Theorem \ref{smalldelta}\ref{smalldelta2}, the derivative $\I_\theta\colon \R^D\rightarrow   L^2_\lambda(\mathcal{X},\R^M) \equiv \R^{M\times M}$ is injective and thus also $\I_\theta^*\I_\theta\in \R^{D\times D}$ is injective. By Remark \ref{nfromi} this equals $\N_\theta$ and as we are in finite dimension, the information operator must be invertible.
\end{proof}

\begin{proof}[Proof of Theorem \ref{bvmthm}]
Under the identification from Remark \ref{vdv}, the conclusion of the theorem equals that of Theorem 10.1 in \cite{vdV98}; it thus remains to check the prerequisites of the latter theorem. The DQM property and invertibility of the information matrix  follow from Proposition \ref{dqm} and  Theorem \ref{informationinvertible}. The separation condition is satisfied in view of Lemma 10.6 in \cite{vdV98} (see also the discussion preceding it), the compactness of $\Theta$ and Theorem \ref{smalldelta}\ref{smalldelta1}, which implies identifiability of the model.
\end{proof}

\noindent {\bf Acknowledgements.} My sincere thanks go to Richard Nickl for several discussions during the preparation of this article and for pointing out Remark \ref{mean}. Further, I would like to thank Giovanni Alberti and Bastian Harrach  for conversations during the AAIP21 workshop in Klagenfurt; these were the starting point of this note. This research was supported by the EPSRC Centre for
Doctoral Training and the Munro--Greaves Bursary of Queens’ College Cambridge.

\bibliographystyle{plain}
\bibliography{fdcalderon2.bbl}

\end{document}